\documentclass[11pt]{amsart}
\usepackage{geometry}                
\usepackage{graphicx}
\usepackage{amssymb,latexsym}
\usepackage{epstopdf}
\theoremstyle{plain}
\newtheorem{theorem}{Theorem}[section]
\newtheorem{corollary}[theorem]{Corollary}

\newtheorem{lemma}[theorem]{Lemma}
\newtheorem{proposition}[theorem]{Proposition}

\theoremstyle{definition}
\newtheorem{definition}[theorem]{Definition}

\theoremstyle{remark}
\newtheorem*{remark}{Remark}

\newtheorem{example}[theorem]{Example}
\title[Essential spectrum, norm, and spectral radius]{The essential spectrum, norm,  and spectral radius of \\ abstract multiplication operators.}
\author{Anton R. Schep}
\address{Department of Mathematics\\
         University of South Carolina\\
	 Columbia, SC 29208}
\email{schep@math.sc.edu}	  
\begin{document}

\begin{abstract}
Let $E$ be a complex Banach lattice and $T$ is an operator in the center $Z(E)=\{T: |T|\le \lambda I \mbox{ for some } \lambda\}$ of $E$. Then the essential norm $\|T\|_{e}$ of $T$ equals the essential spectral radius $r_{e}(T)$ of $T$.  We also prove $r_{e}(T)=\max\{\|T_{A^{d}}\|, r_{e}(T_{A})\}$, where $T_{A}$ is the atomic part of $T$ and $T_{A^{d}}$ is the non-atomic part of $T$. Moreover $r_{e}(T_{A})=\limsup_{\mathcal F}\lambda_{a}$, where $\mathcal F$ is the Fr\'echet filter on the set $A$ of all positive atoms in $E$ of norm one and $\lambda_{a}$ is given by $T_{A}a=\lambda_{a}a$ for all $a\in A$.
\end{abstract}
\maketitle
\section{Introduction.} Over the past two decades numerous authors have studied boundedness, invertibility and compactness of multiplication operators on a variety of spaces of measurable functions on a measure space $(X, \mu)$ (or sequence spaces, when the measure is atomic). We refer to a recent preprint of J. Voigt (\cite{Voigt2022}) for some references. It seems that these authors are not aware that these multiplication operators are examples of order bounded band preserving operators (or orthomorphisms). Zaanen showed in 1977 in \cite{Zaanen1975} already that the collection of all such orthomorphisms is lattice isometric to $L^{\infty}(X, \mu)$,  so that a function defining a multiplication operator has to be essentially bounded and the operator norm of the multiplication operator equals the $L^{\infty}$-norm. Around that time it was also established that for an arbitrary Banach lattice $E$ the collection of  band preserving operators was equal to the center $Z(E)$, where $Z(E)=\{T\in \mathcal L(E): |T|\le \lambda I \mbox{ for some } \lambda\}$ and that $\|T\|=\inf \{\lambda: |T|\le \lambda I\}$ for $T\in Z(E)$. For this and additional basic properties of the center we refer to the books \cite{Meyer-Nieberg1991a} and \cite{Zaanen83}. In particular, Proposition 3.1.3 of \cite{Meyer-Nieberg1991a} shows that for a Dedekind complete Banach lattice the center as defined here coincides with the collection of regular operators which commute with band projections. It was proved in \cite{Schep1980a} that $Z(E)$ is a full sub-algebra of the algebra $\mathcal L(E)$ of norm bounded operators, i.e.,  $T\in Z(E)$ is invertible in $\mathcal L(E)$ if and only if there exists $c>0$ such that $|T|\ge cI$. Therefore $\sigma(T)$ is equal to the essential range of $T$, when $T$ is a multiplication operator on a Banach function space. Moreover the spectral radius $r(T)=\|T\|$. The study of the essential spectrum of operators $T\in Z(E)$ was initiated in \cite{Schep1980a}, where it was proved that $\sigma_{e}(T)=\sigma(T)$ for $T\in Z(E)$ with $E$ a non-atomic Banach lattice. Here the essential spectrum  $\sigma_{e}(T)$ denotes the spectrum of $T+\mathcal K(E)$ in the Calkin algebra $\mathcal L(E)/\mathcal K(E)$. A few years later in \cite{ArendtSourour1986} other spectra and essential spectra of operators $T\in Z(E)$ were considered and an alternative proof of  $\sigma_{e}(T)=\sigma(T)$ for $T\in Z(E)$ with $E$ a non-atomic Banach lattice was given. A description of compact $T\in Z(E)$ can be derived from the more general description of compact disjointness preserving operators due to De Pagter (unpublished) and Wickstead \cite{Wickstead1981}.

More recently the essential norm $\|T\|_{e}$ in the Calkin algebra has been considered for certain concrete  Banach sequence spaces and $L_{p}$-spaces (see \cite{Voigt2022} and \cite{Castillo2021}) for more references. The current paper was prompted by the preprint \cite{Voigt2022}, which didn't use our results from \cite{Schep1980a}. By using the term abstract multiplication operator in our title we hope that future duplication of our results by authors working on these questions for concrete function spaces will be avoided by  a web search or MathScinet search. The current paper describes completely the essential spectrum and norm of an operator $T\in Z(E)$ for an arbitrary Banach lattice $E$. The paper is organized as follows. In section 2 we describe for non-atomic Banach lattices the essential spectrum and norm for an operator $T\in Z(E)$. To be able to describe the essential spectral radius for the atomic case where we might have uncountably many atoms, we describe in section 3 what we call Fr\'echet cluster points of a complex valued function and the Fr\'echet limit superior and inferior of the modulus of such a function. For the countable case this is standard textbook material, but for the uncountable case one can find some of this in General Topology textbooks, but for the
convenience of the reader we have included this section. Then in section 4 the atomic case is handled. In the process of proving the results we reprove the description of a compact $T\in Z(E)$. In the final section we combine the atomic and non-atomic case to describe the essential spectrum and norm of $T\in Z(E)$ for an arbitrary Banach lattice $E$.  As in \cite{Voigt2022} we associate in that case with $T\in Z(E)$ an atomic part $T_{A}$ and non-atomic part $T_{A^{d}}$, but in general (i.e., when $E$ is not Dedekind complete) we don't have a decomposition of $T$, as these parts are not defined on all of $E$. We give in this section a concrete example to illustrate this. We conclude the paper by applying our results to prove that if $T\in Z(E)$ is essentially quasi-nilpotent, then $T$ is compact (and thus $T+\mathcal K(E)=0$ in the Calkin algebra).
\section{The non-atomic case.}
Let $E$ be a complex Banach lattice. Let $A\subset E$. Then the disjoint complement $A^{d}$ of $A$  is defined as $A^{d}=\{x\in E: |x|\wedge |a|=0 \mbox{ for all  } a\in A\}$. Then $A^{dd}$ is defined as $(A^{d})^{d}$. Recall that $a\in E$ is called an atom, if $|x|\le |a|$ implies that $x=\lambda a$ for some scalar $\lambda$, i.e., $\{a\}^{dd}=\{\lambda a:\lambda\in \mathbb C\}$. Now $E$ is called non-atomic if $E$ doesn't contain any atoms. Typical examples of non-atomic Banach lattices are Banach function spaces over a non-atomic $\sigma$-finite measure space. The following theorem was proved by the author in 1980 in \cite{Schep1980a}  (Theorem 1.11).
\begin{theorem}\label{non-atomic case}
Let $E$ be a complex non-atomic Banach lattice and $T\in Z(E)$. Then $\sigma(T)=\sigma_{e}(T)$  and thus $r_{e}(T)=r(T)=\|T\|$.
\end{theorem}
Using the methods of the above mentioned paper we can now prove the following theorem. The proof of this theorem, as well as the following theorem, depend on Lemma 1.10 of \cite{Schep1980a}. As one of the referees correctly pointed out, this lemma is incorrect as stated. Fortunately the fix is easy. In Lemma 1.9 of \cite{Schep1980a} one needs to replace $0\le u_{n}\le u$ by $|u_{n}|\le u$.Then the proof of Lemma 1.9 of \cite{Schep1980a} is correct. Note that in the final step the Eberlein-Smulyan theorem is used. With this correction, Lemma 1.10 of \cite{Schep1980a}  is true if we replace $0\le u_{n} \in B$ by $u_{n}\in B$ in the statement. 
\begin{theorem}
Let $E$ be a complex non-atomic Banach lattice and $T\in Z(E)$. Then $\|T\|_{e}=\|T\|=r_{e}(T)$.
\end{theorem}
\begin{proof}
Let $0<c<\|T\|$. Then $(|T|-cI)^{+}>0$. Therefore there exists $0<u\in E$ such that $v=(|T|-cI)^{+}u>0$. Now $(|T|-cI)v=((|T|-cI)^{+})^{2}u\ge 0$. Therefore $|T|v\ge cv$. By order continuity of $T$ this implies that $|Tw|
\ge c|w|$ for all $w\in \{v\}^{dd}$. In particular $\|Tw\|\ge c\|w\|$ for all $w\in \{v\}^{dd}$.  Now by the corrected (as indicated above) Lemma 1.10 of \cite{Schep1980a} there exist $w_{n}\in \{v\}^{dd}$ with $\|w_{n}\|=1$ such that $w_{n}\to 0$ in the weak topology $\sigma(E, E^{*})$. Let now $K\in \mathcal K(E)$. Then, by compactness of $K$, we have that $\|Kw_{n}\|\to 0$ as $n\to \infty$. From this it follows that 
\begin{align*}
\|T-K\|&\ge \limsup \|Tw_{n}-Kw_{n}\|\\
&\ge \limsup |\|Tw_{n}\|-\|Kw_{n}\||=\limsup \|Tw_{n}\|\ge c.
\end{align*}
As this holds for all $0<c<\|T\|$, it follows that $\|T-K\|\ge \|T\|$. Hence $\|T\|_{e}\ge \|T\|$ and thus $\|T\|_{e}=\|T\|$. The result follows now from the above theorem.
\end{proof}
\begin{remark}
In case $E^{*}$ is non-atomic, it was already proved in \cite{Schep1989} that $\|T\|_{e}=\|T\|$ for the larger class of disjointness preserving operators.
\end{remark}
\section{Fr\'echet cluster points and the Fr\'echet limit superior. }
Let $A$ be a non-empty set and $f:A\to \mathbb C$ a function. Denote by $\mathcal F$ the Fr\'echet filter on $A$, i.e., a subset $B$ of $A$ is in $\mathcal F$ if $B^{c}$ is a finite set. 
\begin{definition} A point $z\in \mathcal C$ is called a Fr\'echet cluster point of $f$ if for all $\epsilon>0$ and all $B\in \mathcal F$ there is an element $a\in B$ such that $|f(a)-z|<\epsilon$. 

\end{definition}
In this case we will also say that $z$ is an $\mathcal F$-cluster point. Note that one can equivalently say that $z$ is a Fr\'echet cluster point of $f$ if and only if $z$ is a cluster point of $f(\mathcal F)$ in the topological sense if and only if there exist distinct $a_{n}\in A$ ($n=1,2, \cdots$) such that $f(a_{n})\to z$.
One can show (assuming the axiom of choice) the following proposition.
\begin{proposition} Let $A$ be a non-empty set and $f:A\to \mathbb C$ a function. Then $z$ is a Fr\'echet cluster point of $f$ if and only there exist a free ultrafilter $\mathcal U$ on $A$ such that $z=\lim_{\mathcal U}f$.

\end{proposition}
It is also known that the set of all $\mathcal F$-cluster points of $f$ is a closed subset of $\mathcal C$. Therefore we have the following proposition.
\begin{proposition} Let $A$ be a non-empty set and $f:A\to \mathbb C$ a bounded function. Then there exists an $\mathcal F$-cluster point of $f$ with largest and smallest modulus. 

\end{proposition}
We now describe the modulus this largest and smallest $\mathcal F$-cluster point.
\begin{definition}
Let $A$ be a non-empty set and $f:A\to \mathbb R$ a bounded function. Then the $\mathcal F$-limit superior of $f$ is defined as 
\begin{equation*}
\limsup_{\mathcal F} f=\lim_{F\in \mathcal F}\sup f[F]=\inf_{F\in \mathcal F}\sup f[F]
\end{equation*}
and the $\mathcal F$-limit inferior of $f$ is defined as
\begin{equation*}
\liminf_{\mathcal F} f=\lim_{F\in \mathcal F}\inf f[F]=\sup_{F\in \mathcal F}\inf f[F].
\end{equation*}

\end{definition}
Note that these notions can be extended to unbounded function by allowing $+\infty$ or $-\infty$ as values.
As in the sequential case one can prove now the following theorem.
\begin{theorem}\label{liminf} Let $A$ be a non-empty set and $f:A\to \mathbb R$ a bounded function. Then the $\mathcal F$-limit superior of $f$ is the largest $\mathcal F$-cluster point of $f$ and $\liminf_{\mathcal F} f$ is the smallest $\mathcal F$-cluster point of $f$.

\end{theorem}
\begin{corollary}\label{largest}Let $A$ be a non-empty set and $f:A\to \mathbb C$ a bounded function with $\lambda$  an $\mathcal F$-cluster point of $f$ of largest modulus. Then
\begin{equation*}
|\lambda|=\limsup_{\mathcal F} |f|.
\end{equation*}

\end{corollary}
\begin{proof} The inequality $| |z_{1}|-|z_{2}| |\le | z_{1}-z_{2}|$ implies immediately that if $z$ is an $\mathcal F$-cluster point of $f$, then $|z|$ is an $\mathcal F$-cluster point of $|f|$. Conversely, if $x$ is an $\mathcal F$-cluster point of $|f|$, then there exist distinct $a_{n}$ in $A$ such that $|f(a_{n})|\to x$. By passing to a subsequence we can assume that there exists $z$ such that $f(a_{i})\to z$. Then $|z|=x$. This shows that $x\in \mathbb R$ is an $\mathcal F$-cluster point of $f$ if and only there exist a an $\mathcal F$-cluster point $z$ of $f$ with $|z|=x$. From this we  the corollary follows immediately.

\end{proof}
We indicate in the next section a case, where $\limsup_{\mathcal F}|f|$ naturally occurs as the quotient norm of $\ell_{\infty}(A)$ by $c_{0}(A)$. 

\section{The atomic case.}
Let $E$ be an infinite dimensional (complex) Banach lattice throughout this section. Recall that $a\in E$ is called an atom, if the band generated by $a$ is one dimensional, i.e., $\{a\}^{dd}=\{\lambda a:\lambda\in \mathbb C\}$. Denote by $A$ the set of all positive atoms in $E$ of norm one. Then $E$ is called an atomic Banach lattice if the band $A^{dd}$ generated by $A$ equals $E$. If we denote by $P_{a}$ the band projection from $E$ onto $\{a\}^{dd}$, then every $x\in E$ can be written as an order convergent series 
\begin{equation*}
x=\sum_{a\in A}P_{a}x.
\end{equation*}
Note that if $A$ is countable, then we have the usual Banach sequence spaces.
Now every $T\in Z(E)$ can be represented as a multiplication operator, where for each $a\in A$ there exists $\lambda_{a}\in \mathbb C$ such that $Ta=\lambda_{a}a$.  Then, if $x=\sum_{a\in A}P_{a}x$, we have $Tx=\sum_{a\in A}\lambda_{a}P_{a}x$.
The following lemma is well-known in the countable case. We include for convenience of the reader the similar proof for the uncountable case.

\begin{lemma} Let $E$ be an atomic Banach lattice with, as above, the set $A$ the set of all positive atoms of norm one. Let $T\in Z(E)$ as above. Then $T\in \mathcal K(E)$ if and only if
\begin{equation*}
\lim_{\mathcal F}|\lambda_{a}|=\limsup_{\mathcal F}|\lambda_{a}|=0.
\end{equation*}
\begin{proof}
Assume first that $T\in \mathcal K(E)$. Then there exists $\epsilon>0$ such that $\limsup_{\mathcal F}|\lambda_{a}|>\epsilon>0$. Then there exist distinct $a_{n}\in A$, for $n=1, 2, \cdots$, such that $|\lambda_{a_{n}}|\ge \epsilon$ for all $n$. As $T$ is a compact operator we can assume, by passing to a subsequence, that $Ta_{n}=\lambda_{a_{n}}a_{n}\to x\in E$. As $|\lambda _{a_{n}}|\le \|T\|$ we can, by passing to a further subsequence, assume that $\lambda_{a_{n}}\to \lambda_{0}$, where $\lambda_{0}\neq 0$. Now $\lambda_{0}a_{n}=(\lambda_{0}-\lambda_{a_{n}})a_{n}+\lambda_{a_{n}}a_{n}\to x$, which implies that $a_{n}\to \frac 1{\lambda_{0}}x$. This contradicts that $\|a_{n}-a_{m}\|=\|a_{n}+a_{m}\|\ge \|a_{n}\|=1$ for all $n\neq m$ and thus $\limsup_{\mathcal F}|\lambda_{a}|=0$. This implies that $\liminf_{\mathcal F}|\lambda_{a}|=\limsup_{\mathcal F}|\lambda_{a}|=0$ and thus $\lim_{\mathcal F}|\lambda_{a}|=0$. Conversely if $\limsup_{\mathcal F}|\lambda_{a}|=0$. Then for all $\epsilon>0$ we have $|\lambda_{a}|\ge \epsilon$ for at most finitely many $a\in A$. This implies that $\{a\in A: \lambda_{a}\neq 0\}$ is countable. Let $\{a_{n}: n=1, 2, \cdots\}=\{a\in A: \lambda_{a}\neq 0\}$ (with the obvious modification if the set on the right is finite). We have then that $\lim_{n\to \infty}\lambda_{a_{n}}=0$. Now the estimate
\begin{equation*}
\|T-\sum_{n=1}^{N}\lambda_{a_{n}}P_{a_{n}}\|\le \sup_{n\ge N+1}|\lambda_{a_{n}}|
\end{equation*}
implies that $T$ is a norm limit of finite rank operators and thus compact.
\end{proof}

\end{lemma}
We now show that  the expression $\limsup_{\mathcal F}|\lambda_{a}|$ is actually a quotient norm. Let $\ell_{\infty}(A)$ denote the Banach lattice of all bounded functions $f:A\to \mathbb C$ with the supremum norm and $c_{0}(A)$ the closed ideal of all $f\in \ell_{\infty}(A)$ with $\limsup_{\mathcal F}|\lambda_{a}|=0$.
Then 
\begin{theorem} Let $\ell_{\infty}(A)$ and  $c_{0}(A)$ be as above. Then we have  $\|f+c_{0}(A)\|=\limsup_{\mathcal F}|f|$ for all $f\in \ell_{\infty}(A)$.

\end{theorem}
\begin{proof}
Note first that
\[\|f+c_{0}(A)\|=\inf\{\sup_{a}|f(a)-g(a)|: \lim_{\mathcal F}g(a)=0\}.\]
Let $\epsilon>0$. Then $\|f+c_{0}(A)\|<\|f+c_{0}(A)\|+\epsilon$ implies that there exists $g\in c_{0}(A)$ such that $\sup_{a}|f(a)-g(a)|<\|f+c_{0}(A)\|+\epsilon$. Now $g\in c_{0}(A)$ implies that there is a finite subset $B$ of $A$ such that $|g(a)|<\epsilon$ for all $a\in A\setminus B$. This implies that $\sup_{a}|f(a)|<\|f+c_{0}(A)\|+2\epsilon$ for all $a\in A\setminus B$. Hence
\[\limsup_{\mathcal F}|f|\le \|f+c_{0}(A)\|.\]
Similarly for $\epsilon>0$ there exists a finite subset $B$ of $A$ such that $|f(a)|<\limsup_{\mathcal F}|f|+\epsilon$ for all $a\in A\setminus B$. Now define $g\in c_{0}(A)$ by $g(a)=f(a)$ for $a\in B$ and $g(a)=0$ otherwise. Then $\sup_{a}|f(a)-g(a)|< \limsup_{\mathcal F}|f|+\epsilon$. Hence 
\[\|f+c_{0}(A)\|\le \limsup_{\mathcal F}|f|,\]
which completes the proof.
\end{proof}
\begin{corollary}  Let $E$ be an atomic Banach lattice with, as above, the set $A$ the set of all positive atoms of norm one. Let $T\in Z(E)$. Then 
\[\inf\{\|T-K\|: K\in Z(E)\cap\mathcal K(E)\}=\limsup_{\mathcal F}|\lambda_{a}|.\]

\end{corollary}
\begin{theorem}\label{atomic case} Let $E$ be an atomic Banach lattice with, as above, the set $A$ the set of all positive atoms of norm one. Let $T\in Z(E)$. Then the essential spectrum $\sigma_{e}(T)$ is given by
\begin{equation*}
\sigma_{e}(T)=\{\lambda\in \mathbb C: \lambda \mbox{ is  a Fr\'echet cluster point of the function }a\mapsto \lambda_{a}\}.
\end{equation*}
Moreover, the essential spectral radius $r_{e}(T)$ of $T$ is given by
\begin{equation*}
r_{e}(T)=\limsup_{\mathcal F}|\lambda_{a}|,
\end{equation*}
and the essential norm $\|T\|_{e}$ of $T$ satisfies $\|T\|_{e}=r_{e}(T)$.

\end{theorem}
\begin{proof} We note that it is sufficient to prove $0\in \sigma_{e}(T)$ if and only if $0$ is a Fr\'echet cluster point of the function $a\mapsto \lambda_{a}$, as $T\in Z(E)$ if and only $T-\lambda I\in Z(E)$.
Assume first that $0\in\sigma_{e}(T)$ and that $0$ is not a Fr\'echet cluster point of  the function $a\mapsto \lambda_{a}$. Then by Theorem \ref{liminf} we have $\liminf_{\mathcal F}|\lambda_{a}|>0$. Let $0<\epsilon<\liminf_{\mathcal F}|\lambda_{a}|$. Then there exists a finite subset $B$ of $A$ such that $|\lambda_{a}|\ge \epsilon$ for all $a\in A\setminus B$. Now $\sum_{a\in B}\lambda_{a}P_{a}$ is a finite rank operator, so that $\sigma_{e}(T)=\sigma_{e}(T-\sum_{a\in B}\lambda_{a}P_{a})$. Now $|\lambda_{a}|\ge \epsilon$ for all $a\in A\setminus B$ implies that $|T-\sum_{a\in B}\lambda_{a}P_{a}|\ge \epsilon \sum_{a\in A\setminus B}P_{a}$. As $I=\sum_{a\in B}P_{a}+\sum_{a\in A\setminus B}P_{a}$ and as $\sum_{a\in B}P_{a}$ is a band projection, it follows that also $\sum_{a\in A\setminus B}P_{a}$ is a band projection. Combined with the lower estimate we obtained above, this implies that $T-\sum_{a\in A\setminus B}\lambda_{a}P_{a}$ is invertible on $\sum_{a\in A\setminus B}P_{a}(E)$. This implies that $T-\sum_{a\in A\setminus B}\lambda_{a}P_{a}$ is a Fredholm operator on $E$ and thus $0\notin  \sigma_{e}(T-\sum_{a\in B}\lambda_{a}P_{a})=\sigma_{e}(T)$, which contradicts our assumption. Now assume that  $0$ is a Fr\'echet cluster point of the function $a\mapsto \lambda_{a}$. Then there exist infinitely many distinct $a_{n}$ with $n=1,2, \cdots$ such that $\lambda_{a_{n}}\to 0$. This implies that $0\in \sigma(T)$ and that $0$ is not isolated in $\sigma(T)$. It is well-known that this implies that $0\in \sigma_{e}(E)$, which completes the proof of the first part of the theorem. The formula for the essential spectral radius follows now immediately from Corollary \ref{largest}. The formula for the essential norm follows from the above and the previous corollary as follows
\[\|T\|_{e}=\inf\{\|T-K\|: K\in \mathcal K\}\le \inf\{\|T-K\|: T\in Z(E)\cap\mathcal K(E)\}=\limsup_{\mathcal F}|\lambda_{a}|=r_{e}(T).\]
 \end{proof}
 \section{The general case.}
 Let $E$ be a (complex) Banach lattice.  Denote by $A$ the set of all positive atoms in $E$ of norm one. Then the atomic part of $E$ is the band $A^{dd}$ generated by $A$ and will be denoted by $E_{A}$. The disjoint complement $A^{d}$ of $A$ will be called the non-atomic part (or continuous part) of $E$ and will be denoted by $E_{A^{d}}$. Below we shall see by means of an example that these bands  in  general don't need to be projection bands (but are so of course when $E$ is Dedekind complete). Therefore we have in general that $E_{A}\oplus E_{A^{d}}$ is only a norm closed, order dense ideal in $E$. We denote the restriction of $T\in Z(E)$ to $E_{A}$ by $T_{A}$ (the atomic part of $T$) and to $E_{A^{d}}$ by $T_{A^{d}}$ (the non-atomic or continuous part of $T$). As $T$ is band preserving we have $T_{A}\in Z(E_{A})$ and $T_{A^{d}}\in Z(E_{A^{d}})$, but in general $T$ is not the direct sum of these two operators, as $T_{A}\oplus T_{A^{d}}$ is only defined on $E_{A}\oplus E_{A^{d}}$. To overcome this technical difficulty, we use the following lemma.


\begin{proposition} Let $E$ be a Banach lattice and $F\subset E$  a norm closed, order dense ideal in $E$. Let $T\in Z(E)$. Then $\|T_{F}\|=\|T\|$.
\end{proposition}
\begin{proof}
As $T_{F}$ is a restriction of $T$ we always have $\|T_{F}\|\le \|T\|$. To prove the reverse inequality, observe that
\[\|T_{F}\|=\inf\{\lambda: |T_{F}|\le \lambda I_{F}\}.\]
Now $|T_{F}|\le \lambda I_{F}$ implies by order continuity of $T$ that also $|T|\le \lambda I$. This implies that $\|T\|\le \|T_{F}\|$ and equality of norms follows.
\end{proof}
\begin{theorem}\label{essential} Let $E$ be a (complex) Banach lattice and $T\in Z(E)$. Then 
\[\sigma_{e}(T)=\sigma_{e}(T_{A})\cup \sigma(T_{A^{d}})\]
and thus
\[r_{e}(T)=\max\{r_{e}(T_{A}), \|T_{A^{d}}\|\}.\]
Moreover $\|T\|_{e}=r_{e}(T)$.

\end{theorem}
\begin{proof}
It is easy to see that $\sigma_{e}(T_{A})\subset \sigma_{e}(T)$ and $\sigma_{e}(T_{A^{d}})\subset \sigma_{e}(T)$, so $\sigma_{e}(T_{A})\cup \sigma_{e}(T_{A^{d}})\subset \sigma_{e}(T)$. Therefore assume that $\lambda\notin \sigma_{e}(T_{A})\cup \sigma_{e}(T_{A^{d}})$. Then $T_{A^{d}}-\lambda I_{A^{d}}$ is Fredholm on $E_{A^{d}}$. As $\sigma(T_{A^{d}}-\lambda I_{A^{d}})=\sigma_{e}(T_{A^{d}}-\lambda I_{A^{d}})$ by the main result in the non-atomic case, we see that there exists $c_{1}>0$ such that $|T_{A^{d}}-\lambda I_{A^{d}}|\ge c_{1}I_{A^{d}}$. Now $T_{A}-\lambda I_{A}$ is Fredholm on $E_{A}$ implies that there exist a finite subset $B$ of $A$ such that the kernel of  $T_{A}-\lambda I_{A}$ is $P_{B}(E_{A})$. Then $T_{A\setminus B}-\lambda I_{A\setminus B}$ is invertible on $P_{A\setminus B}$ , so there exists $c_{2}>0$ such that $|T_{A\setminus B}-\lambda I_{A\setminus B}|\ge c_{2}I_{A\setminus B}$. This implies that for $c=\min \{c_{1}, c_{2}\}$ that
$|T_{A^{d}\cup A\setminus B}-\lambda I_{A^{d}\cup A\setminus B}|\ge c I_{A^{d}\cup A\setminus B}$. This implies that $T_{A^{d}\cup A\setminus B}-\lambda I_{A^{d}}$ is invertible on $E_{A^{d}\cup A\setminus B}$. As $E_{B}$ is finite dimensional this implies that $\lambda \notin \sigma_{e}(T)$, which shows that $\sigma_{e}(T)=\sigma_{e}(T_{A})\cup \sigma(T_{A^{d}})$. The formula for the essential spectral radius follows now from the previous results. To prove the result about the essential norm, observe that $r_{e}(T)\le \|T\|_{e}$ is always true, so it remains to show that $\|T\|_{e}\le r_{e}(T)$. Let $F=E_{A}\oplus E_{A^{d}}$. Then $F$ is a closed order dense ideal in $E$. Therefore, using the preceding proposition, we have
\begin{align*}
\|T\|_{e}&=\inf \{\|T-K\|: K\in \mathcal K(E)\}\\
&\le \inf \{\|T-K\|: K\in \mathcal K(E)\cap Z(E)\}\\
&=\inf \{\|(T-K)_{F}\|: K\in \mathcal K(E)\cap Z(E)\}\\
&\le \inf\{\max \{\|T_{A^{d}}\|, \|T_{A}-K_{A}\|\} : K\in \mathcal K(E)\cap Z(E)\}\\
&=\max \{\|T_{A^{d}}\|, \inf \{\|T_{A}-K_{A}\| :K\in \mathcal K(E)\cap Z(E)\}\}=r_{e}(T)
\end{align*}
by the above. Hence $\|T\|_{e}=r_{e}(T)$.
\end{proof}
We conclude with an example of a Banach lattice $E$ for which the atomic part $E_{A}$ is not a projection band.\begin{example}
Let $x_{n}=\frac 1{2^{n+1}}$ and $I_{n}=[\frac 1{2^{n+1}+2}, \frac 1{2^{n+1}+1}]$ for $n\ge 1$. Then it is easy to verify that $x_{n}\notin I_{m}$ for $n, m\ge 1$. Now put $K=\cup_{n\ge 1} I_{n}\cup \{x_{n}:n\ge 1\}\cup \{0\}$. Then $K$ is compact in the induced Euclidean topology. Let $E=C(K)$. Let $\delta_{x}$ denote the Dirac delta function supported by $x$. Then it is clear that the set $A$ of positive atoms of norm one equals $\{\delta_{x_{n}}: n\ge 1\}$. Note that $\delta_{0}\notin E$, as $\delta_{0}$ is not continuous at $x=0$. Now the continuous part $E_{A^{d}}$ of $E$ equals $A^{d}=\{f\in E: f(x_{n})=0 \mbox{ for all }n\ge 1\}$. Note $f\in A^{d}$ implies by continuity that $f(0)=0$. Now the atomic part $E_{A}$ is equal to $A^{dd}=\{f\in E: f(x)=0 \mbox{ on } \cup_{n\ge 1}I_{n}\}$. Again by continuity $f(0)=0$ for $f\in E_{A}$. Hence $E_{A}\oplus E_{A^{d}}\neq E$, as the function identical one is in $E$. Therefore $E_{A}$ is not a projection band.
It is well-known that for a $C(K)$ space $E$ each $T\in Z(E)$ is given by a multiplication by a function $p\in C(K)$, as $Z(E)$ is lattice isometric to $C(K)$. Denote by $T_{p}$ the operator $T_{p}f=pf$.Then we have $\sigma(T_{p})=\{p(x): x\in K\}$. Thus we have $r(T_{p})=\|T\|=\|p\|_{\infty}$. Now $T_{p}(\delta_{x_{n}})=p(x_{n})\delta_{x_{n}}$ for all $n\ge 1$. As $0$ is the only cluster point of the sequence $(p(x_{n}))$, it follows that $\sigma_{e}(T_{A})=\{0\}$. The essential spectrum of the non-atomic part of $T_{p}$ is 
\[\sigma_{e}((T_{p})_{A^{d}})=\{p(x): x\in \bigcup_{n=1}^{\infty}I_{n}\cup \{0\}\}.\]
Therefore
\[\|T_{p}\|_{e}=\|(T_{p})_{A^{d}}\|=\max \{|p(x)|: x\in \bigcup_{n=1}^{\infty}I_{n}\cup \{0\}\}.\]
We also observe that $(T_{p})_{A}$ is compact if and only if  we have $p(x_{n})\to 0$, i.e., if and only if $p(0)=0$.  Moreover we can write $T_{p}=T_{p_{1}}+T_{p_{2}}$ with $(T_{p_{1}})_{A}=(T_{p})_{A}$ if and only if $p(0)=0$, i.e., if and only if $(T_{p})_{A}$ is compact. 
\end{example}
As the above example shows that in general, if $T\in Z(E)$, then $T$ is not the  sum of $T_{1}$ and $T_{2}$, where $T_{A}=(T_{1})_{A}$ and $T_{A^{d}}= (T_{2})_{A^{d}}$.However  that is so if only if $(T_{p})_{A}$ is compact. 
We now show that this is always the case for the if part of the statement.
\begin{proposition}  Let $E$ be a (complex) Banach lattice and $T\in Z(E)$ such that $T_{A}$ is compact.
Then there exist $T_{1}, T_{2}\in Z(E)$ with $T_{1}$ compact such that $T= T_{1}+T_{2} $ where   $T_{A}=(T_{1})_{A}$ and $T_{A^{d}}= (T_{2})_{A^{d}}$.
\end{proposition}
\begin{proof}
Let $T_{A}e_{a}=\lambda_{a}e_{a}$. Then $T_{A}$ is compact implies that $\limsup_{\mathcal F}|\lambda_{a}|=0$.  This implies that the set $B=\{a\in A: \lambda_{a}\neq 0\}$ is countable and $T_{A}$ is a norm limit of the series $\sum_{a\in B} \lambda_{a}P_{a}$ in $Z(E_{A})$. As this series also converges in $Z(E)$, it follows that $T_{A}$ extends to a compact operator $T_{1}$  in $Z(E)$ such that $T_{A}=(T_{1})_{A}$. Let $T_{2}=T-T_{1}$. Then it clear from $(T_{1})_{A^{d}}=0$ that $T_{A^{d}}= (T_{2})_{A^{d}}$.
\end{proof}
In \cite{Gustafson1969} it was observed that if $\sigma_{e}(A+B)=\sigma_{e}(A)$ for all bounded self-adjoint operators $A$ on a Hilbert space, for a given self-adjoint $B$, then $B$ is compact. The following theorem is an order analogue of this.
\begin{theorem} Let $E$ be a (complex) Banach lattice and $T\in Z(E)$ such that $\sigma_{e} (I+T)=\sigma_{e}(I)$. Then $T$ is compact.
\end{theorem}
\begin{proof}
From $\sigma_{e} (I+T)=\sigma_{e}(I)$ we conclude that $\sigma_{e}(T)=\{0\}$. From Theorem \ref{essential} we see that $\sigma_{e}(T_{A})=\{0\}$ and $\sigma_{e}(T_{A^{d}})=\{0\}$.
From Theorem \ref{atomic case} we see that $r_{a}(T_{A})=\limsup_{\mathcal F}\lambda_{a}=0$. This implies that $T_{A}$ is compact on $E_{A}$. By the above proposition there exist  $T_{1}, T_{2}\in Z(E)$ with $T_{1}$ compact such that $T= T_{1}+T_{2} $ where   $T_{A}=(T_{1})_{A}$ and $T_{A^{d}}= (T_{2})_{A^{d}}$. Now $\sigma(T_{A^{d}})=\sigma_{e}(T_{A^{d}})=\{0\}$ by Theorem \ref{non-atomic case} and thus $r(T_{A^{d}})=||T_{A^{d}}\|=0$. This implies that $(T_{2})_{A^{d}\oplus A^{dd}}=0$, which by order density of $A^{d}\oplus A^{dd}$ and order continuity of $T_{2}$ implies that $T_{2}=0$. Hence $T=T_{1}$ is compact.
\end{proof}\begin{remark}
We could have phrased the above result alternatively as: If $T\in Z(E)$ is essentially quasi-nilpotent, then $T$ is compact.
\end{remark}

\providecommand{\bysame}{\leavevmode\hbox to3em{\hrulefill}\thinspace}
\providecommand{\MR}{\relax\ifhmode\unskip\space\fi MR }
\providecommand{\MRhref}[2]{%
  \href{http://www.ams.org/mathscinet-getitem?mr=#1}{#2}
}
\providecommand{\href}[2]{#2}

\end{document}